\documentclass[12pt,a4paper]{article}
\usepackage[a4paper,top=2.5cm,bottom=2.5cm,left=2.5cm,right=2.5cm]{geometry}
\usepackage{amsmath}
 \usepackage{graphicx}
\usepackage{amsthm}
\usepackage{color}
\usepackage{fancyhdr}
\usepackage{pdfpages}
\usepackage{txfonts}
\usepackage{amssymb}
\usepackage{paralist}
\usepackage{extarrows}
\usepackage{setspace}
\usepackage{braket} 
\usepackage[colorlinks=true,citecolor=blue,linkcolor=blue,anchorcolor=blue,urlcolor=blue]{hyperref}
\allowdisplaybreaks
\numberwithin{equation}{section}
\newtheorem{theorem}{Theorem}[section]
\newtheorem{lemma}{Lemma} [section]
\newtheorem{corollary}{Corollary} [section]

\begin{document}
\title{Fu\v{c}ik spectrum for  the operator with rapidly increasing weight and applications
\footnote{This study was supported by the National Natural Science Foundation of China (Grant No. 12571117.)}}
\author{Jinzi Bai, Fei Fang \footnote{Corresponding author. Email: fangfei68@163.com} \\  \footnotesize  \emph{School of Mathematics and Statistics,%
Beijing Technology and Business University, Beijing 100048, China}}

\maketitle
\noindent \textbf{\textbf{Abstract:}}  In this paper, we study the Fu\v{c}ik spectrum for the operator with rapidly increasing
weight, which is defined as a set $\Sigma$ comprising those $(\alpha, \beta) \in \mathbb{R}^2$ such that
\begin{equation*}
\left\{\begin{array}{l}
L u:=-\Delta u-\frac{1}{2}(x \cdot \nabla u)=\alpha u^{+}-\beta u^{-}, \text{in}\   \mathbb{R}^N,\\
u\in X,
\end{array}\right.
\end{equation*}
has a non-trivial solution $u$,  where, $N\geq1$, $u^{ \pm}=\max ( \pm u, 0)$, $u=u^{+}-u^{-}$.
The existence of a first nontrivial curve $\mathcal{C}$
 of this spectrum, along with some of its properties (e.g., Lipschitz continuity, strict decrease and asymptotic behavior) is investigated in this paper.
Our difficulty is that the problem  is defined on the whole space $\mathbb{R}^N$,
and therefore certain estimates do not carry over from the Fu\v{c}ik problem on bounded domains.
As an application, we establish the
multiplicity of solutions to the following  problem
\begin{equation*}
\left\{\begin{array}{l}
-\Delta u-\frac{1}{2}(x \cdot \nabla u)=f(x,u), \text{in}\   \mathbb{R}^N,\\
u\in X,
\end{array}\right.
\end{equation*}
where, $N\geq1$ and the nonlinearity $f$ is asymptotically linear at zero and at infinity.



\section{Introduction}
In this study, we investigate the   Fu\v{c}ik spectrum  problem of operator $L$,
which is defined as the set $\Sigma$ of $(\alpha, \beta) \in \mathbb{R}^2$ such that
\begin{equation}\label{ee1}
\left\{\begin{array}{l}
L u:=-\Delta u-\frac{1}{2}(x \cdot \nabla u)=\alpha u^{+}-\beta u^{-}, \ \text{in}\   \mathbb{R}^N \\
u\in X,
\end{array}\right.
\end{equation}
has a nontrivial solution $u$. Here $N\geq1$, $u^{ \pm}=\max ( \pm u, 0)$, $u=u^{+}-u^{-}$ and the definition of space $X$ is provided in the Section 2.
The operator $L$ is closely related to the self-similar solutions of the heat equation, which was studied by
Escobedo and Kavian in \cite{bur6}. The operator $L$ appears in the process of looking for the self-similar solutions
$$
v(t, x)=t^{-1 /(p-2)} u\left(t^{-1 / 2} x\right)
$$
of the heat equation
$$
v_t-\Delta v=|v|^{p-2} v.
$$
Escobedo and Kavian expressed the operator $L$ as the form of a divergence, that is,
$$
L u=-\Delta u-\frac{1}{2}(x \cdot \nabla u)=-\frac{1}{K} \nabla \cdot(K \nabla u),
$$
where $K(x): = {e^{|x{|^2}/4}},$ so the operator $L$  has a variational structure.
They also equipped the operator $L$ with a weighted Sobolev space and proved related embedding theorems  in  \cite{bur6}.

This generalized spectrum concept was proposed by Fu\v{c}ik \cite{bur12} and Dancer \cite{bur11} in the 1970s,
where they considered the problem
\begin{equation*}
-\Delta u=\alpha u^{+}+\beta u^{-} \quad \text { in } \Omega,\left.\quad u\right|_{\partial \Omega}=0.
\end{equation*}
Since then,
several papers have specifically studied the Fu\v{c}ik spectrum of the Laplacian operator on bounded domains $\Omega$
(e.g., \cite{bur4}, \cite{bur13}, \cite{bur14}).
In \cite{bur14},
Schechter proved the existence of the Fu\v{c}ik spectrum near the point $\left(\lambda_k, \lambda_k\right)$,
where $k \in \mathbb{N}$ and $\lambda_k$ is an eigenvalue of $-\Delta$.
In  \cite{bur11}, it was proved that the lines $\lambda_1 \times \mathbb{R}$ and $\mathbb{R} \times \lambda_1$ are isolated in $\Sigma_2$;
while in  \cite{bur13}, the first nontrivial curve in $\Sigma$ passing through $\left(\lambda_2, \lambda_2\right)$ was constructed and characterized via variational methods.
The Fu\v{c}ik spectrum  of the $p$-Laplace operator,
Kirchhoff-type operator, fractional Laplace operator,
 and Schr\"{o}dinger operator have been studied in  \cite{bur4}, \cite{bur8}, \cite{bur9} and \cite{bur10}
 respectively. In \cite{bur10}, \cite{bur4}, \cite{bur9}   and \cite{bur8}, the first non-trivial curve in the
 Fu\v{c}ik point spectrum was constructed using the minimax method,
 and some qualitative characteristics of this curve and its corresponding eigenfunctions were exhibited. As
applications, the authors also obtained some results on existence of multiple solutions for nonlinear
 equations. The reader can be referred to
 \cite{bur15,bur16,bur17,bur19,bur20,bur21,bur22,bur23,bur24,bur18,bur25,bur26,bur27,bur28,bur29,bur30}
for further properties and applications of the Fu\v{c}ik spectrum.


The main purpose of this study is to construct some curves in the Fu\v{c}ik spectrum $\Sigma$ for (\ref{ee1}).
Using a minimax approach, we obtain three curves:
$\left\{\lambda_1\right\} \times \mathbb{R}, \mathbb{R} \times\left\{\lambda_1\right\}$, and
 and $\mathcal{C}=\left\{(p+c(p), c(p)): p\in \mathbb{R}\right\}$, where $c(p)$
with $p \in \mathbb{R}$ is a critical value of the functional $\tilde{I}_p$  as shown in Section 2.
We also study some properties of this curve $\mathcal{C}$, e.g. Lipschitz
continuous, strictly decreasing and asymptotic behavior.
As an application of the Fu\v{c}ik spectrum $\Sigma$, we study the  solvability of the following problem:
\begin{equation}\label{e51}
\left\{\begin{array}{l}
-\Delta u-\frac{1}{2}(x \cdot \nabla u)=f(x,u), \text{in}\   \mathbb{R}^N,\\
u\in X,
\end{array}\right.
\end{equation}
The equation (\ref{e51}) has been studied deeply by using the variational method and obtained certain results
(such as in \cite{bur33,bur31,bur32,bur42,bur34,bur35}  and its references).
 When $f(x, u)=u^{2^{\ast}-1}+\lambda|x|^{\alpha-1} u$, $\alpha>0$, Catrina et al. proved the existence of positive solution for the equation in \cite{bur43}.
 When $f(x, u)=\lambda u+|u|^{p-2} u, 2+2 / N<p<2^{\ast}, \lambda>0$, Haraux and Weissler obtained the existence of radial solution in \cite{bur36} (also see \cite{bur37,bur38}).
 When $f(x, u)=\lambda u+|u|^{p-2} u, \lambda>0$, Naito obtained multiplicity results for radial positive solutions of the equation in \cite{bur39}.
 When $f(x, u)= a(x)|u|^{q-2} u+b(x)|u|^{p-2} u, 1<q<2<p \leq 2^{\ast}$, Furtado et al.
 in \cite{bur40}  showed that the equation has at least two nonnegative nontrivial solutions.
 When $f(x, u)=\lambda|x|^\beta u+|u|^{2^{\ast}-2} u, 2 \leq 2^{\ast}, \lambda>0$, Furtado et al. in \cite{bur41}  obtained sign changing solution.

\section{Constructing curves using minimax methods}
In this section,  following the approach of Cuesta, de Figueiredo and Gossez \cite{bur4},
we use minimax theory to construct three curves belonging to $\Sigma$.
We shall denote by $X$ the Hilbert space obtained as the completion of $C_c^{\infty}\left(\mathbb{R}^N\right)$ with respect to the norm
$$\|u\|:=\left(\int_{\mathbb{R}^N}  |\nabla u|^2 K(x) dx\right)^{\frac{1}{2}}$$
which is induced by the inner product
$$
(u, v):=\int_{\mathbb{R}^N} (\nabla u \cdot \nabla v)K(x) dx.
$$
For each $q \in\left[2,2^*\right]$ we denote by $L_K^q$ the following space
$$
L_K^q:=\left\{u \text { measurable in } \mathbb{R}^N:|u|_{K,q}:=\left(\int_{\mathbb{R}^N} K(x)|u|^q\right)^{1 / q}<\infty\right\} .
$$
Due to the rapid decay at infinity of the functions belonging to $X$ we have the following embedding result proved in \cite{bur6}.
\begin{lemma}\label{le01}
The embedding $X \hookrightarrow L_K^q$ is continuous for all $q \in\left[2,2^*\right]$ and it is compact for all $q \in\left[2,2^*\right)$.
\end{lemma}
Hence, by Lemma \ref{le01}, for every $p \in\left[2,2^{\ast}\right)$, there exists $S_{K,p}>0$ such that
$$
|u|_{K,p} \leqslant S_{K,p}\|u\|, \quad u \in X .
$$
\begin{lemma}[\cite{bur6}]\label{le0}
The eigenvalues of $L$ are the numbers $\left(\lambda_k\right)_{k \geqslant 1}$, where
$$
\lambda_k=\frac{N+k-1}{2}.
$$
The eigenspaces are:
$$
N\left(L-\lambda_k\right)=\operatorname{Span}\left\{D^\beta \varphi_1 ; \quad|\beta|=k-1\right\}
$$
with $\phi_1(x):=\exp \left(-|x|^2 / 4\right), \beta \in N^N,|\beta|=\beta_1+\ldots+\beta_N$,
$
D^\beta=\partial_1^{\beta_1} \ldots \partial_N^\beta.
$
Moreover,
$$
\operatorname{dim} N\left(L-\lambda_k\right)=\binom{N+k-2}{N-1} .
$$
\end{lemma}

We define the nontrivial Fu\v{c}ik spectrum for  problem (\ref{ee1}) as a set $\Sigma_0$ comprising
those $(\alpha, \beta) \in \mathbb{R}^2$ such that (\ref{ee1}) has a sign changing solution.
A function $u \in X$ is a weak solution of (\ref{ee1}), if for every $v\in X$ satisfies
$$
\int_{\mathbb{R}^N}\nabla u \nabla v K(x)=\alpha \int_{\mathbb{R}^N} u^{+} v K(x) d x-\beta \int_{\mathbb{R}^N} u^{-} v K(x) d x.
$$
Weak solutions of (\ref{ee1}) are exactly the critical points of the functional $J: X \rightarrow \mathbb{R}$ defined as
$$
J(u)=\frac{1}{2} \int_{\mathbb{R}^N}|\nabla u|^2 K(x)dx-\frac{\alpha}{2} \int_{\mathbb{R}^N}\left(u^{+}\right)^2 K(x)d x-\frac{\beta}{2} \int_{\mathbb{R}^N}\left(u^{-}\right)^2 K(x)d x.
$$
Then $J$ is Fr\'{e}chet differentiable in $X$ and
$$
\left\langle J^{\prime}(u), \phi\right\rangle=\int_{\mathbb{R}^N}\nabla u \nabla \phi K(x)dx
-\alpha \int_{\mathbb{R}^N} u^{+} \phi K(x) d x-\beta \int_{\mathbb{R}^N} u^{-} \phi K(x)dx.
$$

Let us
consider the following functional:
$$
I_p(u)= \int_{\mathbb{R}^N} |\nabla u|^2 K(x)d x-p \int_{\mathbb{R}^N}\left(u^{+}\right)^2 K(x) d x.
$$
Obviously, $I_p \in C^1(X, \mathbb{R})$ for all $p \in \mathbb{R}$. We seek  the critical points of the restriction of $I_p$ to
$$
\mathcal{P}=\left\{u \in X: \int_{\mathbb{R}^N}u^2 K(x) d x=1\right\}.
$$
Let $\tilde{I}_P=I_P|_\mathcal{P}$. Then $\tilde{I}_P\in  C^1(X, \mathbb{R})$.
We easily  know that $u \in \mathcal{P}$ is a critical point of $\tilde{I}_p$ if and only if there exists $t \in \mathbb{R}$ such that
\begin{equation}\label{ee3}
 \int_{\mathbb{R}^N}\nabla u\nabla v K(x) d x-p \int_{\mathbb{R}^N} u^{+} v K(x)d x=t \int_{\mathbb{R}^N} u v K(x)d x
\end{equation}
for all $v \in X$. Hence $u \in \mathcal{P}$ is a nontrivial weak solution of the problem
\begin{equation}\label{ee35}
\left\{\begin{array}{l}
-L u  =(p+t) u^{+}-t u^{-}, \quad \text { in } \mathbb{R}^N, \\
u\in X.
\end{array}\right.
\end{equation}
This means that $(p+t,t)\in \Sigma$. Setting $u=v$ in \eqref{ee3}, we have  $t=\tilde{I}_p(u)$.
Hence, we have the following result, which describes the relationship between the critical points of $\tilde{I}_p$ and the spectrum $\Sigma$.

\begin{lemma}\label{le1}
 For $p \geq 0,(p+t, t) \in \mathbb{R}^2$ belongs to the spectrum $\Sigma$ if and only if there exists a critical point $u \in \mathcal{P}$ of $\tilde{I}_p$ such that $t=\tilde{I}_p(u)$, a critical value.
\end{lemma}

\begin{lemma}\label{le2}
 Assume that $\phi_1$  is the first eigenfunction  of $L$.
 Then $\phi_1$  is a global minimum for $\tilde{I}_p$ with $\tilde{I}_p\left(\phi_1\right)=\lambda_1-p$. The corresponding point in $\Sigma$ is $\left(\lambda_1, \lambda_1-p\right)$ which lies on the vertical line through $\left(\lambda_1, \lambda_1\right)$.
\end{lemma}
\begin{proof}
We easily know that
$\tilde{I}_p\left(\phi_1\right)=\lambda_1-p$ and
$$
\begin{aligned}
\tilde{I}_p(u) & =\int_{\mathbb{R}^N}|\nabla u|^2 K(x) d x -p \int_{\mathbb{R}^N}\left(u^{+}\right)^2 K(x)d x \\
& \geq \lambda_1 \int_{\mathbb{R}^N} u^2 K(x)d x-p \int_{\mathbb{R}^N}\left(u^{+}\right)^2 K(x)d x \geq \lambda_1-p.
\end{aligned}
$$
So $\phi_1$ is a global minimum of $\tilde{I}_p$ with $\tilde{I}_p\left(\phi_1\right)=\lambda_1-p$.

\end{proof}
\begin{lemma}\label{le3}
   The negative eigenfunction $-\phi_1$ is a strict local minimum for $\tilde{I}_p$ with $\tilde{I}_p\left(-\phi_1\right)=\lambda_1$.
    The corresponding point in $\Sigma$ is $\left(\lambda_1+p, \lambda_1\right)$, which lies on the horizontal line through $\left(\lambda_1, \lambda_1\right)$.
\end{lemma}
\begin{proof}
We assume by contradiction  that there exists a sequence $u_k \in \mathcal{P}$, $u_k \neq-\phi_1$ with $\tilde{I}_p\left(u_k\right) \leq \lambda_1$, $u_k \rightarrow-\phi_1$ in $X$. Firstly, we claim  that $u_k$ changes $\operatorname{sign}$ for sufficiently large $k$.  Indeed, since $u_k \rightarrow-\phi_1$,
then we have $u_k < 0$ somewhere for the sufficiently large $k$.
Moreover, if $u_k \leq 0$ for a.e $x \in \mathbb{R}^N$, then
$$
\tilde{I}_p\left(u_k\right)=\int_{\mathbb{R}^N}|\nabla u_k|^2 K(x) d x>\lambda_1.
$$
  Since $u_k \neq \pm \phi_1$ and we get contradiction as $\tilde{I}_p\left(u_k\right) \leq \lambda_1$. So $u_k$ changes sign for sufficiently large $k$. Define $w_k:=\frac{u_k^{+}}{\left|u_k^{+}\right|_{K,2}}$ and
$$
r_k:=\int_{\mathbb{R}^N}|\nabla w_k|^2  K(x) d x.
$$
Now we show that $r_k \rightarrow \infty$. We assume by contradiction  that $r_k$ is bounded.
Then there exists a subsequence of $w_k$ still denoted by $w_k$ and $w \in X$ such that $w_k \rightharpoonup w$ weakly in $X$ and $w_k \rightarrow w$ strongly in $L_K^2\left(\mathbb{R}^n\right)$.
Hence $\int_{\mathbb{R}^N} w^2 K(x)d x=1, w \geq 0$ a.e. on $\mathbb{R}^N$
and so for some $\varepsilon>0$, $\delta=\left|\left\{x \in \mathbb{R}^N: w(x) \geq \varepsilon\right\}\right|>0$.
 As $u_k \rightarrow-\phi_1$ in $X$ and hence in $L_{K}^2(\mathbb{R}^N)$.
 Therefore $\left|\left\{x \in \mathbb{R}^N: u_k(x) \geq \varepsilon\right\}\right| \rightarrow 0$ as $k \rightarrow \infty$ and so $\left|\left\{x \in \mathbb{R}^N: w_k(x) \geq \varepsilon\right\}\right| \rightarrow 0$
 as $k \rightarrow \infty$ which is a contradiction as $\delta>0$. Hence, $r_k \rightarrow \infty$. Now we have
\begin{equation}\label{ee4}
  \begin{aligned}
\tilde{I}_p(u_k) & =\int_{\mathbb{R}^N}|\nabla u_k|^2 K(x) d x-p \int_{\mathbb{R}^N}\left(u_k^{+}\right)^2 K(x) d x \\
& =\int_{\mathbb{R}^N}|\nabla (u^{+}_k-u^{-}_k)|^2 K(x) d x-p \int_{\mathbb{R}^N}\left(u_k^{+}\right)^2 K(x) d x\\
& =  \int_{\mathbb{R}^N}|\nabla u^{+}_k|^2K(x) d x +\int_{\mathbb{R}^N} |\nabla u^{-}_k|^2 K(x) d x-p \int_{\mathbb{R}^N}\left(u_k^{+}\right)^2 K(x) d x\\
& \geq  (r_k-p) \int_{\mathbb{R}^N}\left(u_k^{+}\right)^2 K(x) d x +\lambda_1\int_{\mathbb{R}^N} ( u^{-}_k)^2 K(x) d x.
\end{aligned}
\end{equation}
Since  $u_k \in \mathcal{P}$, we obtain
\begin{equation}\label{ee5}
  \tilde{I}_p\left(u_k\right) \leq \lambda_1=\lambda_1 \int_{\mathbb{R}^N}\left(u_k^{+}\right)^2 K(x)d x+\lambda_1 \int_{\mathbb{R}^N}\left(u_k^{-}\right)^2 K(x)d x.
\end{equation}
Using \eqref{ee4} and \eqref{ee5}, we have,
$$
\left(r_k-p\right) \int_{\mathbb{R}^N}\left(u_k^{+}\right)^2 K(x) d x+\lambda_1 \int_{\mathbb{R}^N}\left(u_k^{-}\right)^2 d x \leq \lambda_1 \int_{\mathbb{R}^N}\left(u_k^{+}\right)^2 K(x) d x+\lambda_1 \int_{\mathbb{R}^N}\left(u_k^{-}\right)^2 K(x)d x .
$$
Hence, $$\left(r_k-p-\lambda_1\right) \int_{\mathbb{R}^N}\left(u_k^{+}\right)^2 K(x)d x \leq 0.$$
Then $r_k-p \leq \lambda_1$, which contradicts the fact that $r_k \rightarrow+\infty$.
\end{proof}

\begin{lemma}\label{le4}(\cite{bur2})
Let $E$ be a Banach space, $g, f \in C^1(E, \mathbb{R}), M=\{u \in E \mid g(u)=1\}$ and $u_0, u_1 \in M$. Let $\varepsilon>0$ such that $\left\|u_1-u_0\right\|>\varepsilon$ and
$$
\inf \left\{f(u): u \in M \text { and }\left\|u-u_0\right\|_E=\varepsilon\right\}>\max \left\{f\left(u_0\right), f\left(u_1\right)\right\} .
$$
Assume that $f$ satisfies the PS condition on $M$ and that
$$
\Gamma=\left\{\gamma \in C([-1,1], M): \gamma(-1)=u_0 \text { and } \gamma(1)=u_1\right\}
$$
is non empty. Then $$c=\inf _{\gamma \in \Gamma} \max _{u \in \gamma[-1,1]} f(u)$$ is a critical value of $\left.f\right|_M$.
\end{lemma}

\begin{lemma}\label{le5}
$\tilde{I}_p$  satisfies the PS condition on $\mathcal{P}$.
\end{lemma}
\begin{proof}
 Let $\left\{u_k\right\}$ be a PS sequence,  that is,  there exists $C>0$ and $t_k$ such that
\begin{equation}\label{ee6}
   \left|\tilde{I}_p\left(u_k\right)\right| \leq C,
\end{equation}
and
\begin{equation}\label{ee7}
   \int_{\mathbb{R}^N} \nabla u_k\nabla v K(x)dx -p \int_{\mathbb{R}^N} u_k^{+} vK(x)dx -t_k \int_{\mathbb{R}^N} u_k vK(x)dx=o_k(1)\|v\|.
\end{equation}
From (\ref{ee6}), we have that  $u_k$ is bounded in $X$.
So we can  suppose that up to a subsequence $u_k \rightharpoonup u_0$ weakly in $X$,
and $u_k \rightarrow u_0$ strongly in $L_K^2(\mathbb{R}^N)$.
Letting $v=u_k$ in (\ref{ee7}), we get $t_k$ is bounded and up to a subsequence $t_k$ converges to $t$.
We next show that $u_k \rightarrow u_0$ strongly in $X$.
Since $u_k \rightharpoonup u_0$ weakly in $X$, we have
\begin{equation}\label{ee712}
   \int_{\mathbb{R}^N} \nabla u_k\nabla v K(x)dx\rightarrow  \int_{\mathbb{R}^N} \nabla u_0\nabla v K(x)dx.
\end{equation}
 for all $v \in X$, and
$$
\tilde{I_p^{\prime}}\left(u_k\right)\left(u_k-u_0\right)=o_k(1).
$$
Hence, we deduce
\begin{align}\label{ee8}
   &\left|\int_{\mathbb{R}^N} |\nabla u_k|^2 K(x)dx- \int_{\mathbb{R}^N} \nabla u_k\nabla u_0 K(x)dx\right|\nonumber\\
     &\leq p \left|\int_{\mathbb{R}^N} u_k^{+} \left(u_k-u_0\right) K(x)dx\right| +t_k \left|\int_{\mathbb{R}^N} u_k \left(u_k-u_0\right)K(x)dx\right|+o_k(1)\nonumber\\
     &\leq   o_k(1)+p\left|u_k^{+}\right|_{L_K^2}\left|u_k-u_0\right|_{L_K^2}+\left|t_k\right|\left|u_k\right|_{L_K^2}\left|u_k-u_0\right|_{L_K^2} \rightarrow 0 \quad \text { as } k \rightarrow \infty .
\end{align}
By setting $v=u_0$ in  (\ref{ee712}), we obtain
$$
\int_{\mathbb{R}^N} \nabla u_k \nabla u_0 K(x) d x \rightarrow \int_{\mathbb{R}^N} |\nabla u_0|^2 K(x) d x.
$$
Combining both the inequalities we have,
$$
\int_{\mathbb{R}^N} |\nabla u_k|^2 K(x) d x \rightarrow \int_{\mathbb{R}^N} |\nabla u_0|^2 K(x) d x.
$$
Thus
$$
\left\|u_k-u_0\right\|^2=\left\|u_k\right\|^2+\left\|u_0\right\|^2-2 \int_{\mathbb{R}^N} \nabla u_k \nabla u_0 K(x) d x\rightarrow0, \text { as } k \rightarrow \infty,
$$
that is,  $u_k$ $\rightarrow$ $u_0$ strongly in $X$.
\end{proof}

\begin{lemma}\label{le6}
Let $\varepsilon_0>0$ be such that
\begin{equation}\label{ee9}
 \tilde{I}_p(u)>\tilde{I}_p\left(-\phi_1\right)
\end{equation}
for all $u \in B\left(-\phi_1, \varepsilon_0\right) \cap \mathcal{P}$ with $u \neq-\phi_1$, where the ball is chosen in $X$. Then for any $0<\varepsilon<\varepsilon_0$,
\begin{equation}\label{ee10}
 \inf \left\{\tilde{I}_p(u): u \in \mathcal{P} \text { and }\left\|u-\left(-\phi_1\right)\right\|=\varepsilon\right\}>\tilde{I}_p\left(-\phi_1\right) .
\end{equation}

\end{lemma}
\begin{proof}
  Assume by contradiction that the infimum in (\ref{ee10}) is equal to $\tilde{I}_p\left(-\phi_1\right)= \lambda_1$
  for some $\varepsilon$ with $0<\varepsilon<\varepsilon_0$.
  Then there exists a sequence $u_k \in \mathcal{P}$ with $\left\|u_k-\left(-\phi_1\right)\right\|=\varepsilon$ such that
$$
\tilde{I}_p\left(u_k\right) \leq \lambda_1+\frac{1}{2 k^2} .
$$
Define the set $$A=\left\{u \in \mathcal{P}: \varepsilon-\delta \leq\left\|u-\left(-\phi_1\right)\right\| \leq \varepsilon+\delta\right\},$$
 where $\delta$ is chosen such that $\varepsilon-\delta>0$ and $\varepsilon+\delta<\varepsilon_0$.
 Using the  contradiction hypothesis and (\ref{ee9}),
we have
 $$\inf \left\{\tilde{I}_p(u): u \in A\right\}=\lambda_1.$$
 Now for each $k$, we use Ekeland's variational principle(see \cite{bur3}) to the functional $\tilde{I}_p$ on $A$ to obtain the existence of $v_k \in A$ such that
\begin{equation}\label{ee11}
\begin{aligned}
& \tilde{I}_p\left(v_k\right) \leq \tilde{I}_p\left(u_k\right), \quad\left\|v_k-u_k\right\| \leq \frac{1}{k}, \\
& \tilde{I}_p\left(v_k\right) \leq \tilde{I}_p(u)+\frac{1}{k}\left\|u-v_k\right\| \text{\ for\ all}\  u \in A.
\end{aligned}
\end{equation}

Our aim is to prove  that $v_k$ is a PS sequence for $\tilde{I}_p$ on $\mathcal{P}$
 i.e. $\tilde{I}_p\left(v_k\right)$ is bounded and $\left\|\tilde{I}_p^{\prime}\left(v_k\right)\right\|_{\ast}\rightarrow 0$.
 Once this is proved we get by Lemma \ref{le5}, up to a subsequence $v_k \rightarrow v$ strongly in $X$.
Obviously $v \in \mathcal{P}$ and satisfies
 $\left\|v-\left(-\phi_1\right)\right\| \leq \varepsilon+\delta<\varepsilon_0$ and $\tilde{I}_p(v)=\lambda_1$
 which contradicts the given hypotheses. Clearly, $\tilde{I}_p\left(v_k\right)$ is  bounded.
 So we only need to prove that $\left\|\tilde{I}_p^{\prime}\left(v_k\right)\right\|_{\ast} \rightarrow 0$.
 Now, fix $k>\frac{1}{\delta}$ and take $w \in X$ tangent to $\mathcal{P}$ at $v_k$ i.e
\begin{equation}\label{ee12}
 \int_{\mathbb{R}^N} v_k w K(x)d x=0
\end{equation}
and for $t \in \mathbb{R}$, define
\begin{equation}\label{ee13}
u_t:=\frac{v_k+t w}{\left|v_k+t w\right|_{K,2}} .
\end{equation}

We first observe that for $|t|$ sufficiently small, $u_t \in A$. Indeed
$$
\lim _{t \rightarrow 0}\left\|u_t-\left(-\varphi_1\right)\right\|=\left\|v_k-\left(-\varphi_1\right)\right\|,
$$
where the right-hand side is
$$
\leqslant\left\|v_k-u_k\right\|+\left\|u_k-\left(-\varphi_1\right)\right\| \leqslant \frac{1}{k}+\varepsilon<\delta+\varepsilon
$$
and it is also
$$
\geqslant\left\|\left(-\varphi_1\right)-u_k\right\|-\left\|u_k-v_k\right\| \geqslant \varepsilon-\frac{1}{k}>\varepsilon-\delta .
$$
Hence, we can take $u=u_t$ in (\ref{ee11}) and
put $r(t)=\left|v_k+t w\right|_{K,2}^2$. For the sufficiently small $t>0$, it holds that
\begin{equation}\label{ee14}
\frac{1}{t}\left[\tilde{I}_p\left(v_k\right)-\tilde{I}_p\left(v_k+t w\right)\right] \leqslant \frac{1}{r(t)} \frac{1}{t}[1-r(t)] \tilde{I}_p\left(v_k+t w\right)+\frac{1}{n[r(t)]^{1 / 2}}\left\|t^{-1}\left[1-[r(t)]^{1 / 2}\right] v_k+w\right\| .
\end{equation}

The first term in the right-hand side of (\ref{ee14}) involves $\frac{\left(r(t)-1\right)}{t}$, which is the differential quotient of $r(t)$ near $t=0$.
Since, by (\ref{ee12}),
$$
\left.\frac{d}{d t} r(t)\right|_{t=0}= \int_{\mathbb{R}^N} v_k w K(x)=0,
$$
we have that $\frac{\left(r(t)-1\right)}{t}  \rightarrow 0$ as $t \rightarrow 0$, and we deduce that the first term in the right-hand side of (\ref{ee14}) goes to zero as $t \rightarrow 0$. The second term in the right-hand side of (\ref{ee14}) involves $(1-r(t)^\frac{1}{2})t^{-1}$, which also goes to zero as
$t\rightarrow0$ (again by (\ref{ee12})).
By going to the limit in (\ref{ee14}) as $t \rightarrow 0$, we find that $\left\langle \tilde{I}_p^{\prime}\left(v_k\right), w\right\rangle \leqslant\|w\| / k$. Consequently,
\begin{equation}\label{ee15}
\left|\left\langle \tilde{I}_p^{\prime}\left(v_k\right), w\right\rangle\right| \leq \frac{1}{k}\|w\|,
\end{equation}
for all $w\in X$ tangent to $\mathcal{P}$ at $v_k$.
Since $w$ is arbitrary in $X$, we choose $\alpha_k$ such that
\begin{equation}\label{ee151}
  \int_{\mathbb{R}^N} v_k\left(w-\alpha_k v_k\right) K(x)d x=0.
\end{equation}
Replacing $w$ by $w-\alpha_k v_k$ in (\ref{ee15}), we have
$$
\left|\left\langle \tilde{I}_p^{\prime}\left(v_k\right), w\right\rangle-\alpha_k\left\langle \tilde{I}_p^{\prime}\left(v_k\right), v_k\right\rangle\right| \leq \frac{1}{k}\left\|w-\alpha_k v_k\right\|.
$$
Since $\left\|\alpha_k v_k\right\| \leq C\|w\|$, by \eqref{ee151}, we get
$$
\left|\left\langle \tilde{I}_p^{\prime}\left(v_k\right), w\right\rangle-t_k \int_{\mathbb{R}^N} v_k w K(x)d x\right| \leq \frac{C}{k}\|w\|,
$$
where $t_k=\left\langle \tilde{I}_p^{\prime}\left(v_k\right), v_k\right\rangle$. Hence
$$
\left\|\tilde{I}_p^{\prime}\left(v_k\right)\right\|_{\ast} \rightarrow 0 \quad \text { as } k \rightarrow \infty,
$$
and $v_k$ is a PS
sequence for $\tilde{I}_p$  on $\mathcal{P}$.
\end{proof}
\begin{lemma}\label{le7}
Let $\varepsilon>0$ such that $\left\|\phi_1-\left(-\phi_1\right)\right\|>\varepsilon$ and
$$
\inf \left\{\tilde{I}_p(u): u \in \mathcal{P} \text { and }\left\|u-\left(-\phi_1\right)\right\|=\varepsilon\right\}>\max \left\{\tilde{I}_p\left(-\phi_1\right), \tilde{I}_p\left(\phi_1\right)\right\} .
$$
Then $\Gamma=\left\{\gamma \in C([-1,1], \mathcal{P}): \gamma(-1)=-\phi_1\right.$ and $\left.\gamma(1)=\phi_1\right\}$ is nonempty and
\begin{equation}\label{ee16}
 c(p)=\inf _{\gamma \in \Gamma} \max _{u \in \gamma[-1,1]} \tilde{I}_p(u)
\end{equation}
is a critical value of $\tilde{I}_p$. Moreover $c(p)>\lambda_1$.
\end{lemma}

\begin{proof}
  Let $\phi \in X$ be such that $\phi \notin \mathbb{R} \phi_1$ and consider the path
   $$\gamma(t)= \frac{t \phi_1+(1-|t|) \phi}{\left|t \phi_1+(1-|t|) \phi\right|_{K,2}},$$
   then $\gamma(t) \in \Gamma$. Moreover by Lemma \ref{le5} and Lemma \ref{le6}, $\tilde{I}_p$
   satisfies (P.S) condition and the geometric assumptions. Then by  Lemma \ref{le7},
   $c(p)$ is a critical value of $\tilde{I}_p$.
   Using the definition of $c(p)$ we have $c(p)> \max \left\{\tilde{I}_p\left(-\phi_1\right), \tilde{I}_p\left(\phi_1\right)\right\}=\lambda_1$.
\end{proof}

We have thus proved the
following

\begin{theorem}\label{t1}
 For each $p \geqslant 0$, the point $(p+c(p), c(p))$, where $c(p)>\lambda_1$ is defined by the minimax formula (\ref{ee16}), belongs to $\Sigma$.
\end{theorem}

Now, we can easily obtain the following important result:
$$
\tilde{I}_{-p}(-u)=\tilde{I}_p(u)+p, \quad u \in \mathcal{P}, p \in \mathbb{R} .
$$
Indeed, if we set $u \in \mathcal{P}, p \in \mathbb{R}$, then we get
\begin{equation}\label{ee165}
  \tilde{I}_{-p}(-u)=\|u\|^2+p\left|(-u)^{+}\right|_{K,2}^2=\|u\|^2-p\left|(-u)^{-}\right|_{K,2}^2+p=\|u\|^2-p\left|u^{+}\right|_{K,2}^2+p=\tilde{I}_p(u)+p.
\end{equation}

\begin{lemma}\label{le91}
$c(-p)=c(p)+p$ for all $p \in \mathbb{R}$.
\end{lemma}
\begin{proof}
 Without loss of generality, we suppose that $p \in \mathbb{R}_{+}$; if not,
  we replace $p$ by $-p$ in the following process.
 For each $\gamma \in \Gamma$, we have $-\gamma \in \Gamma^{-}$. Based on the definition of $c(-p)$ and (\ref{ee165}), we have
$$
c(-p) \leqslant \max _{t \in[-1,1]} \tilde{I}_{-p}(-\gamma(t))=\tilde{I}_{-p}\left(-\gamma\left(t_0\right)\right)=\tilde{I}_p\left(\gamma\left(t_0\right)\right)+p \leqslant \max _{t \in[-1,1]} \tilde{I}_p(\gamma(t))+p .
$$
It follows that $c(-p) \leqslant c(p)+p$. Similarly, for each $\gamma \in \Gamma^{-}$, we obtain $-\gamma \in \Gamma$. Then,
$$
c(p) \leqslant \max _{t \in[-1,1]} \tilde{I}_p(-\gamma(t))=\tilde{I}_p\left(-\gamma\left(t_1\right)\right)=\tilde{I}_{-p}\left(\gamma\left(t_1\right)\right)-p \leqslant \max _{t \in[-1,1]} \tilde{I}_{-p}(\gamma(t))-p.
$$
Thus, $c(p) \leqslant c(-p)-p$. Hence, $c(-p)=c(p)+p$.
\end{proof}

\section{First nontrivial curve}
This section is devoted to the construction of the first nontrivial curve in the Fu\v{c}ik point spectrum of $L$.
 For the sake of simplicity, we denote a cross curve by $$\mathcal{L}=\left\{(a, b) \in \mathbb{R}^2:\left(a-\lambda_1\right)\left(b-\lambda_1\right)=0\right\}=\left(\left\{\lambda_1\right\} \times \mathbb{R}\right) \cup\left(\mathbb{R} \times\left\{\lambda_1\right\}\right),$$
 and two broken lines by
 $$\mathcal{L}_1=\left\{(a, b) \in \mathbb{R}^2: \max \{a, b\}=\lambda_1\right\}=\left(\left\{\lambda_1\right\} \times\left(-\infty, \lambda_1\right]\right) \cup\left(\left(-\infty, \lambda_1\right] \times\left\{\lambda_1\right\}\right)$$
 and
 $$\mathcal{L}_2=\left\{(a, b) \in \mathbb{R}^2:\right. \left.\min \{a, b\}=\lambda_1\right\}=\left(\left\{\lambda_1\right\} \times\left[\lambda_1, \infty\right)\right) \cup\left(\left[\lambda_1, \infty\right) \times\left\{\lambda_1\right\}\right).$$

\begin{theorem}\label{t2}
Let $P \in \mathbb{R}$. Then, $c(p)=\min \left\{\beta:(p+\beta, \beta) \in \Sigma_0\right\}$.
\end{theorem}

Before giving the proof of Theorem \ref{t2}, we need the following results.

\begin{lemma}\label{le75}
The eigenfunctions of the operator $L$ corresponding to the non-first eigenvalue $\lambda$ change sign.
\end{lemma}

\begin{proof}
 Let $\phi_1, u$ be the eigenfunctions corresponding to $\lambda_1$ and $\lambda$ respectively. Then $\phi_1, u$ satisfies
\begin{equation}\label{ez1}
  L \phi_1:=-\Delta \phi_1-\frac{1}{2}(x \cdot \nabla \phi_1)=\lambda_1 \phi_1,
\end{equation}
\begin{equation}\label{ez2}
  L u:=-\Delta u-\frac{1}{2}(x \cdot \nabla u)=\lambda u,
\end{equation}
respectively. Suppose $u$ does not change sign. We may assume $u \geq 0$ in $\mathbb{R}^N$. Let $\left\{\psi_n\right\}$ be a sequence in $C_c^{\infty}$ such that $\psi_n \rightarrow \phi_1$ as $n \rightarrow \infty$. Now consider the test functions
$$w_1=\phi_1, w_2=\frac{\psi_n^2}{u+\frac{1}{n}}.$$
Then $w_1, w_2 \in X$. Testing (\ref{ez1}) with $w_1$ and (\ref{ez2}) with $w_2$ we get
\begin{equation}\label{ez3}
  \int_{\mathbb{R}^N}|\nabla \phi_1|^2 K(x) d x-\lambda_1 \int_{\mathbb{R}^N}\phi_1^2 K(x) d x=0,
\end{equation}

\begin{equation}\label{ez4}
\int_{\mathbb{R}^N} \nabla u \cdot \nabla\left(\frac{\psi_n^2}{u+\frac{1}{n}}\right) K(x)d x
-\lambda\int_{\mathbb{R}^N} \psi_n^2\left(\frac{u}{u+\frac{1}{n}}\right) K(x)d x=0.
\end{equation}
By adopting the method described in \cite[Theorem 1.1]{bur5}, we have
\begin{equation}\label{ez5}
\int_{\mathbb{R}^N}|\nabla \psi_n|^2 K(x) d x
-\lambda\int_{\mathbb{R}^N} \psi_n^2\left(\frac{u}{u+\frac{1}{n}}\right) K(x)d x\geq0.
\end{equation}
Subtracting (\ref{ez3}) from (\ref{ez5}) and taking the limit as $n\rightarrow \infty$ we get,
$$
\left(\lambda-\lambda_1\right) \int_{\mathbb{R}^N}  \phi_1^2 K(x)dx \leq 0.
$$
This is a contradiction to the fact that $\lambda > \lambda_1$.

\end{proof}

\begin{lemma}\label{le76}
Let $\left\{u_n\right\} \subset X \backslash\{0\}$ be a nonnegative sequence and
$\left|\Omega_n\right|:=\left|\left\{x \in \mathbb{R}^N: u_n(x)>0\right\}\right| \rightarrow 0$ as $n \rightarrow \infty$.
 Then, $\left\|u_n\right\| /\left|u_n\right|_{K,q} \rightarrow \infty$ for all $q \in\left[1,2^*\right)$.
\end{lemma}
\begin{proof}
Let $w_n=u_n /\left|u_n\right|_{K,q}$ for all $n$ and assume by contradiction that there exists $w \in X$,
for a subsequence still denoted by $\left\{w_n\right\}$,
such that $w_n \rightarrow w$ in $L_{K}^q(\mathbb{R}^N)$. According to \cite[Lemma A.1, p.133]{bur7},
there  exist a subsequence still denoted by $\left\{w_n\right\}$ and $w_0 \in L_{K}^q(\mathbb{R}^N)$ such that
$\left|w_n(x)\right| \leqslant w_0(x)$, a.e. $x \in \mathbb{R}^N$. For any given $\varepsilon>0$, since
$$
\left\{x \in \mathbb{R}^N: w_n(x) \geqslant \varepsilon\right\} \subset \Omega_n,
$$
we find that $\left\{w_n\right\}$ converges to 0 in measure. According to Lebesgue's dominated convergence theorem, it follows that $1=\left|w_n\right|_{K,q}^q \rightarrow 0$, which is impossible.
\end{proof}
\begin{corollary}\label{le77}
(i) For any nontrivial signed solution $u$ of (\ref{ee1}), it holds that $(\alpha, \beta) \in \mathcal{L}$;

(ii) For any $(\alpha, \beta) \in \mathcal{L}$, if $u$ is a solution of (\ref{ee1}) corresponding to $(\alpha, \beta)$, then it must be signed. Furthermore, if $(\alpha, \beta)=\left(\lambda_1, \lambda_1\right)$, then $u=  \phi_1$, and if $(\alpha, \beta) \in \mathcal{L} \backslash\left\{\left(\lambda_1, \lambda_1\right)\right\}$, then $$u=\left(\operatorname{sgn}\left|\beta-\lambda_1\right|-\right. \left.\operatorname{sgn}\left|\alpha-\lambda_1\right|\right) \phi_1.$$
\end{corollary}

\begin{proof}
  Using Lemma \ref{le75}, we can easily obtain this conclusion.
\end{proof}
\begin{corollary}\label{le78}
For any sign changing solution of (\ref{ee1}),  the point $(\alpha, \beta)$ is on the upper right of the broken line $\mathcal{L}_2$, i.e., $\alpha>\lambda_1$ and $\beta>\lambda_1$.
\end{corollary}

\begin{proof}
  Using Corollary \ref{le77}, we can easily obtain this conclusion.
\end{proof}

\begin{lemma}\label{le8}
(i) $\Sigma$ is closed in $\mathbb{R}^2$.

(ii) $\Sigma_0$ is also closed in $\mathbb{R}^2$. Moreover, $\bar{\Sigma}_0 \cap \mathcal{L}=\emptyset$.
\end{lemma}
\begin{proof}
(i) Let a sequence $\left\{\left(\alpha_n, \beta_n\right)\right\} \subset \Sigma$ satisfy $\left(\alpha_n, \beta_n\right) \rightarrow(\alpha, \beta)$ as $n \rightarrow \infty$. We prove that $(\alpha, \beta) \in \Sigma$. First, we know that there exists a sequence $\left\{u_n\right\} \subset \mathcal{P}$ such that

\begin{equation}\label{ee17}
\left\{\begin{array}{l}
L u_n:=-\Delta u_n-\frac{1}{2}(x \cdot \nabla u_n)=\alpha_n u_n^{+}-\beta_n u_n^{-}, \\
u_n\in X .
\end{array}\right.
\end{equation}
Obviously, $\left\{u_n\right\}$ is bounded in $X$.
Hence, we may assume that by passing to a subsequence, if necessary, still denoted by $\left\{u_n\right\}$
that $u_n \rightharpoonup u \in X$.
 Thus, $|u|_{K,2}=1$ by $\left\{u_n\right\} \subset \mathcal{P}$ and Lemma \ref{le01}.
 Denote a functional $I_{\alpha_n, \beta_n} \in C^1(X, \mathbb{R})$ by
$$
I_{\alpha_n, \beta_n}(v)=\|v\|^2-\alpha_n\left|v^{+}\right|_{K,2}^2+\beta_n\left|v^{-}\right|_{K,2}^2, \quad v \in X .
$$
From $\left\langle I_{\alpha_n, \beta_n}^{\prime}\left(u_n\right), u_n-u\right\rangle=0, u_n \rightarrow u$ in $L_K^2(\mathbb{R}^N)$
and the boundedness of $\left\{u_n\right\},\left\{\alpha_n\right\}$ and $\left\{\beta_n\right\}$, we conclude that $u_n \rightarrow u$ in $X$.
Consequently, $u$ is a nontrivial solution of (\ref{ee1}), which means that $(\alpha, \beta) \in \Sigma$.

(ii) Let a sequence $\left\{\left(\alpha_n, \beta_n\right)\right\} \subset \Sigma_0$ satisfy $\left(\alpha_n, \beta_n\right) \rightarrow(\alpha, \beta)$ as $n \rightarrow \infty$. We obtain $(\alpha, \beta) \in \Sigma$ by (i). We claim that $(\alpha, \beta) \notin \mathcal{L}$, and thus the result holds by $\Sigma=\mathcal{L} \cup \Sigma_0$. Arguing by contradiction, we assume that $(\alpha, \beta) \in \mathcal{L}$. Let $u_n$ and $u$ be nontrivial solutions of (\ref{ee1}) corresponding to $\left(\alpha_n, \beta_n\right)$ and $(\alpha, \beta)$, respectively.

Then, $u$ is signed and $u=- \phi_1$ if $u \leqslant 0, u= \phi_1$ if $u \geqslant 0$ by Corollary \ref{le77}.
This implies that
$\left|u_n>0\right|:=\left|\left\{x \in \mathbb{R}^N: u_n(x)>0\right\}\right| \rightarrow 0$
or $\left|u_n<0\right|:=\left|\left\{x \in \mathbb{R}^N: u_n(x)<0\right\}\right| \rightarrow 0$. Since $u_n$ is sign changing for all $n$, by using Lemma \ref{le76} for $u_n^{+}$ and $-u_n^{-}$, we derive
$$
\alpha_n= \frac{\left\|u_n^{+}\right\|^2}{\left|u_n^{+}\right|_{K,2}^2} \rightarrow \infty
$$
or
$$
\beta_n= \frac{\left\|u_n^{-}\right\|^2}{\left|u_n^{-}\right|_{K,2}^2} \rightarrow \infty,
$$
which is a contradiction.
\end{proof}

 Similar to \cite[Lemma 3.5]{bur4}, the following lemma is apparent.
\begin{lemma}\label{le9}
There are some topological properties on $\mathcal{P}$; in particular,

(i) $\mathcal{P}$ is locally arcwise connected;

(ii) any connected open subset $O$ of $\mathcal{P}$ is arcwise connected;

(iii) if $O_1$ is a component of an open set $O \subset \mathcal{P}$, then $\partial O_1 \cap O=\emptyset$.
\end{lemma}

\begin{lemma}\label{le10}
 Define $O=\left\{u \in \mathcal{P}: \tilde{I}_p(u)<r\right\}$. Then any component of $O$ contains a critical point of $\tilde{I}_p$ on $\mathcal{P}$.
\end{lemma}
\begin{proof}
  Let $O_1$ be a component of $O$ and let $d=\inf \left\{I_p(u): u \in \bar{O}_1\right\}<r$.
  We prove that the infimum $d$ is achieved at some $u \in \bar{O}_1$.
  In fact, let $\left\{u_n\right\} \subset \bar{O}_1$ be a minimizing sequence that satisfies
 $$I_p\left(u_n\right) \leqslant d+1 / n^2$$ for all $n$.
 From Ekeland's variational principle, it follows that there exists a sequence $\left\{v_n\right\} \subset \bar{O}_1$  such that
\begin{equation}\label{ee18}
 \tilde{I}_p\left(v_n\right) \leqslant \tilde{I}_p\left(u_n\right),
\end{equation}
\begin{equation}\label{ee19}
 \left\|v_n-u_n\right\| \leqslant \frac{1}{n},
\end{equation}
\begin{equation}\label{ee20}
\tilde{I}_p\left(v_n\right) \leqslant \tilde{I}_p(v)+\frac{1}{n}\left\|v-v_n\right\|, \quad v \in \bar{O}_1.
\end{equation}
By (\ref{ee18}), we derive that $\left\{v_n\right\}$ is bounded in $X$. Moreover, we take $w \in T_{v_n} \mathcal{P}$ for the fixed $n$ with $d+1 / n^2<r$, i.e.,
$$
\int_{\mathbb{R}^N} v_n w K(x)dx=0,
$$
and consider $u_t$ as defined by (\ref{ee13}). We observe that $v_n \in O_1$. If $v_n \in \partial O_1$, then $\tilde{I}_p\left(v_n\right)=r$ from Lemma \ref{le9}, which contradicts
$$
\tilde{I}_p\left(v_n\right) \leqslant \tilde{I}_p\left(u_n\right) \leqslant d+1 / n^2<r .
$$
Hence, $u_t \in O_1$ for the sufficiently small $|t|$. By taking $v=u_t$ in (\ref{ee20}) and using the same method in Lemma \ref{le6}, we see that $\left\{v_n\right\}$ is a PS sequence for $\tilde{I}_p$ on $\mathcal{P}$. According to Lemma \ref{le5}, the claim holds. If $u \in \partial O_1$, then this is a contradiction by Lemma \ref{le9}. Hence, $u \in O_1$.
\end{proof}

Now we provide the proof of Theorem \ref{t2}.
\begin{proof} [Proof of Theorem \ref{t2}]
  First, we consider the case where $p \in \mathbb{R}_{+}$. By contradiction,
  assume that there exists some $\beta \in\left(\lambda_1, c(p)\right)$  such that $(p+\beta, \beta) \in \Sigma_0$.
   By Lemma \ref{le8}, we can choose $b \in\left(\lambda_1, \beta\right]$ such that $b$ is a critical value of $\tilde{I}_p$ on $\mathcal{P}$,
  and there is no critical value in ( $\lambda_1, b$ ). In the following, we construct a path $\gamma \in \Gamma$ that satisfies $\tilde{I}_p(\gamma(t)) \leqslant b$
  for all $t \in[-1,1]$, which yields a contradiction of the definition of $c(p)$.

  Let $u \in \mathcal{P}$ be a critical point of $\tilde{I}_p$ at level $b$. Then, $u$ satisfies the equation
\begin{equation}\label{ee21}
\left\{\begin{array}{l}
L u=-\Delta u-\frac{1}{2}(x \cdot \nabla u)=(p+b) u^{+}-b u^{-}, \\
u\in X .
\end{array}\right.
\end{equation}
and we know that $u$ changes sign in $\mathbb{R}^N$ by Corollary \ref{le77}. From this equation, it holds that
$$
\left\|u^{+}\right\|^2=(p+b)\left|u^{+}\right|_{K,2}^2
$$
and
$$
\left\|u^{-}\right\|^2=b\left|u^{-}\right|_{K,2}^2.
$$
Consequently, we see that

$$
\tilde{I}_p\left(\frac{u^{+}}{\left|u^{+}\right|_{K,2}}\right)=b, \quad \tilde{I}_p\left(\frac{u^{-}}{\left|u^{-}\right|_{K,2}}\right)=b-p, \quad \tilde{I}_p\left(\frac{-u^{-}}{\left|u^{-}\right|_{K,2}}\right)=b.
$$
Now, we consider the following path $\gamma_1$ in $\mathcal{P}$, which tends from $u$ to $u^{-} /\left|u^{-}\right|_{K,2}$,
\begin{equation}\label{ee211}
 \gamma_1(t)=\frac{(1-t) u+tu^{-}}{\left|(1-t) u+tu^{-}\right|_{K,2}}, \quad t \in[0,1].
\end{equation}
We can show that $\tilde{I}_p\left(\gamma_1(t)\right) \leqslant b$ for all $t \in[0,1]$. In fact, a simple calculation shows that
$$
(1-t) u+tu^{-}=(1-t) u^{+}+(2 t-1) u^{-}, \quad t \in[0,1]
$$
and
$$
\left[(1-t) u+t\left(u^{-}\right)\right]^{+}= \begin{cases}(1-t) u^{+}+(2 t-1) u^{-}, & t \in[0,1 / 2], \\  (1-t) u^{+} , & t \in[1 / 2,1].\end{cases}
$$
Hence, for $t \in[1 / 2,t]$,
$$
\begin{aligned}
\tilde{I}_p\left(\gamma_1(t)\right) & =\left\|\gamma_1(t)\right\|^2-p \left|\left[\gamma_1(t)\right]^{+}\right|_{K,2} ^2 \\
& =\frac{\left\|(1-t) u^{+}+(2 t-1) u^{-}\right\|^2-p\left|(1-t) u^{+}\right|_{K,2}^2}{\left|(1-t) u^{+}+(2 t-1) u^{-}\right|_{K,2}^2} \\
& = \frac{(1-t)^2\left\|u^{+}\right\|^2+(2 t-1)^2\left\|u^{-}\right\|^2-p(1-t)^2\left|u^{+}\right|_{K,2}^2}{(1-t)^2\left|u^{+}\right|_{K,2}^2
+(2 t-1)^2\left|u^{-}\right|_{K,2}^2} \\
& =\frac{b(1-t)^2\left|u^{+}\right|_{K,2}^2+b(2 t-1)^2\left|u^{-}\right|_{K,2}^2}{(1-t)^2\left|u^{+}\right|_{K,2}^2+(2 t-1)^2\left|u^{-}\right|_{K,2}^2}=b .
\end{aligned}
$$
Similarly, for $t \in[0,1 / 2]$,
$$
\begin{aligned}
\tilde{I}_p\left(\gamma_1(t)\right) & =\left\|\gamma_1(t)\right\|^2-p \left|\left[\gamma_1(t)\right]^{+}\right|_{K,2} ^2 \\
& =\frac{\left\|(1-t) u^{+}+(2 t-1) u^{-}\right\|^2-p\left|(1-t) u^{+}+(2 t-1) u^{-}\right|_{K,2}^2}{\left|(1-t) u^{+}+(2 t-1) u^{-}\right|_{K,2}^2} \\
& = \frac{(1-t)^2\left\|u^{+}\right\|^2+(2 t-1)^2\left\|u^{-}\right\|^2-p(1-t)^2\left|u^{+}\right|_{K,2}^2-p(2t-1)^2\left|u^{-}\right|_{K,2}^2}{(1-t)^2\left|u^{+}\right|_{K,2}^2
+(2 t-1)^2\left|u^{-}\right|_{K,2}^2} \\
& =\frac{b(1-t)^2\left|u^{+}\right|_{K,2}^2+(b-p)(2 t-1)^2\left|u^{-}\right|_{K,2}^2}{(1-t)^2\left|u^{+}\right|_{K,2}^2+(2 t-1)^2\left|u^{-}\right|_{K,2}^2}\leq b .
\end{aligned}
$$
Therefore, we can start from $u$ to $u^{-} /\left|u^{-}\right|_{K,2}$ along $\gamma_1$ and the energy is not higher than $b$ on this path.
In order to continue, we must study the levels below $b-p$.
Let $O=\left\{v \in \mathcal{P}: \tilde{I}_p(v)<b-p\right\}$. Obviously, $\phi_1 \in O$, when $-\phi_1 \in O$ if $b-p>\lambda_1$.
Furthermore, it is easy to see that $\phi_1$ and $-\phi_1$ are the only possible critical points of $\tilde{I}_p$ in $O$.
Indeed, if we assume that $v \in O$ is a critical point of $\tilde{I}_p$ on $\mathcal{P}$,
then we have $c:=\tilde{I}_p(v) \leqslant \lambda_1$ based on the choice of $b$.
Firstly, we consider $c<\lambda_1$. Using $(p+c, c) \in \Sigma$,
 Corollary \ref{le77} and Corollary \ref{le78} it follows that $p+c=\lambda_1$.
Thus, according to Corollary \ref{le77} again, $v=\phi_1$. Second,
if we assume that $c=\lambda_1$, then $\left(p+\lambda_1, \lambda_1\right) \in \mathcal{L}$.
By combining with Corollary \ref{le77}, we deduce $v= \pm \phi_1$ when $p=0$, and $v=-\phi_1$ when $p>0$. Therefore, this statement holds true.

Since $u^{-} /\left|u^{-}\right|_{K,2} \in O$, which does not change sign and vanishes on a set of positive measure,
we see that $u^{-} /\left|u^{-}\right|_{K,2} \neq \pm \phi_1$ and it is not a critical point of $\tilde{I}_p$ on $\mathcal{P}$.
From Lemma \ref{le8} and Lemma \ref{le9},
we know that there is a path $\gamma_2$ in $O$ from $u^{-} /\left|u^{-}\right|_{K,2}$ to $\phi_1$ or $-\phi_1$,
and the energy along this path is lower than $b-p$.
Let us assume that it is $\phi_1$ and the other case is similar.

Finally, we construct a path from $-\phi_1$ to $\phi_1$ where the energy is not higher than $b$.
Clearly, $-\gamma_2$ goes from $-u^{-} /\left|u^{-}\right|_{K,2}$ to $-\phi_1$. For all $w \in \mathcal{P}$,
$$
\left|\tilde{I}_p(w)-\tilde{I}_p(-w)\right| \leqslant p,
$$
so we obtain
$$
\tilde{I}_p\left(-\gamma_2(t)\right) \leqslant \tilde{I}_p\left(\gamma_2(t)\right)+p \leqslant b, \quad t \in[0,1] .
$$
Hence, the inverse path $\left(-\gamma_2\right)^{-1}$ goes from $-\phi_1$ to $-u^{-} /\left|u^{-}\right|_{K,2}$ and the energy is not higher than $b$ on this path. Now, let
\begin{equation}\label{ee212}
 \gamma_3(t)=\frac{-(1-t) u^{-}+t u}{\left|-(1-t) u^{-}+t u\right|_{K,2}}, \quad t \in[0,1] .
\end{equation}
Clearly, the path $\gamma_3$ brings us back from $-u^{-} /\left|u^{-}\right|_{K,2}$ to $u$
and the energy is not higher than $b$ on this path according to a similar method
to that mentioned earlier.
By combining everything together in proper order $\left(-\gamma_2\right)^{-1}, \gamma_3, \gamma_1, \gamma_2$,
 we have constructed a path $\gamma_4$ in $\mathcal{P}$ from $-\phi_1$ to $\phi_1$
 where the energy is not higher than $b$. Finally, if we let $\gamma(t)=\gamma_4((t+1) / 2)$
  for all $t \in[-1,1]$, then $\gamma \in \Gamma$ is the desired path. Hence, the result holds for $p \in \mathbb{R}_{+}$.

Now, we show that the result is still correct for $p \in(-\infty, 0)$. In fact,
we suppose by contradiction that there is a $\beta \in\left(\lambda_1-p, c(p)\right)$
 such that $(p+\beta, \beta) \in \Sigma_0$. According to the symmetry of $\Sigma_0$,
it follows that $(\beta, p+\beta) \in \Sigma_0$. Since $\lambda_1<p+\beta<p+c(p)=c(-p)$
by Lemma \ref{le91}, for $-p>0$, we deduce that there exists $p+\beta \in\left(\lambda_1, c(-p)\right)$
 such that $(-p+(p+\beta), p+\beta) \in \Sigma_0$, which contradicts the conclusion obtained above.
\end{proof}

\section{Some properties of the curve $\mathcal{C}$}

Let
$$
\mathcal{C}:=\{(p+c(p), c(p)): p \in \mathbb{R}\}.
$$
In this section, we prove that the curve $\mathcal{C}$ is Lipschitz continuous, exhibits a certain asymptotic behavior, and is strictly decreasing.
We  prove the following three main results.

\begin{theorem}\label{t41}
The map $(p+c(p), c(p)), p \in \mathbb{R}$ is Lipschitz continuous on $\mathbb{R}$.
\end{theorem}

\begin{theorem}\label{t42}
Let $p_1, p_2 \in \mathbb{R}$ with $p_1<p_2$. Then, $p_1+c\left(p_1\right)<p_2+c\left(p_2\right)$ and $c\left(p_1\right)>c\left(p_2\right)$.
\end{theorem}

\begin{theorem}\label{t43}
$c(p) \rightarrow \lambda_1$ as $p \rightarrow \infty$.
\end{theorem}

\begin{proof} [Proof of Theorem \ref{t41}]
   Let $p_1, p_2 \in \mathbb{R}$ and $p_1<p_2$. It is easy to see that
   $c\left(p_1\right) \geqslant c\left(p_2\right)$ from $\tilde{I}_{p_1}(u) \geqslant \tilde{I}_{p_2}(u)$ for all $u \in \mathcal{P}$
   and the definition of $c(\cdot)$. Firstly, we consider the case where $p_1$ and $p_2$ have the same signs.
   Without loss of generality, suppose that $p_1, p_2 \in \mathbb{R}_{+}$. Using the definition
   of $c\left(p_2\right)$ again, for each $\varepsilon>0$, there exists $\gamma \in \Gamma$ such that
$$
\max _{t \in[-1,1]} \tilde{I}_{p_2}(\gamma(t)) \leqslant c\left(p_2\right)+\varepsilon / 2 .
$$
Assume that there is a $t_0 \in[-1,1]$ satisfies
$$
\tilde{I}_{p_1}\left(\gamma\left(t_0\right)\right)=\max _{t \in[-1,1]} \tilde{I}_{p_1}(\gamma(t)) .
$$
By combining the facts given above and noting that $\gamma\left(t_0\right) \in \mathcal{P}$, we have
$$
\begin{aligned}
0 & \leqslant c\left(p_1\right)-c\left(p_2\right) \\
& \leqslant \max _{t \in[-1,1]} \tilde{I}_{p_1}(\gamma(t))-\max _{t \in[-1,1]} \tilde{I}_{p_2}(\gamma(t))+\varepsilon / 2 \\
& \leqslant \tilde{I}_{p_1}\left(\gamma\left(t_0\right)\right)-\tilde{I}_{p_2}\left(\gamma\left(t_0\right)\right)+\varepsilon / 2 \\
& =\left(p_2-p_1\right)\left|\left[\gamma\left(t_0\right)\right]^{+}\right|_{K,2}^2+\varepsilon / 2 \\
& \leqslant\left(p_2-p_1\right)+\varepsilon / 2.
\end{aligned}
$$
Similarly, for $p_1, p_2 \in \mathbb{R}_{-}$, it still holds that $0 \leqslant c\left(p_1\right)-c\left(p_2\right) \leqslant\left(p_2-p_1\right)+\varepsilon / 2$. Second, assume that $p_1<0<p_2$. Then, from the results given above, we deduce that
$$
0 \leqslant c\left(p_1\right)-c\left(p_2\right)=c\left(p_1\right)-c(0)+c(0)-c\left(p_2\right) \leqslant-p_1+\varepsilon / 2+p_2+\varepsilon / 2=\left(p_2-p_1\right)+\varepsilon .
$$
Hence, for $p_1, p_2 \in \mathbb{R}$ with $p_1<p_2$, we have
$$
0 \leqslant c\left(p_1\right)-c\left(p_2\right) \leqslant\left(p_2-p_1\right)+\varepsilon,
$$
which shows that $c$ is Lipschitz continuous on $\mathbb{R}$ by the arbitrariness of $\varepsilon$.

\end{proof}
To prove Theorem \ref{t42}, we need the following lemma.
For the sake of simplicity, let $P=\mathbb{R}_{+}^2$. Then, $P$ is a cone in $\mathbb{R}^2$ and its interior $P^{\circ}=(0, \infty)^2$. For $x, y \in \mathbb{R}^2$, we define $x \leqslant y$ iff $y-x \in P$, and $x<y$ if $x \leqslant y$ with $x \neq y$. We also define $x \ll y$ iff $y-x \in P^{\circ}$.

\begin{lemma}\label{le41}
 Let $(\alpha, \beta) \in \mathcal{C}$. If $(a, b) \in \mathbb{R}^2$ satisfies $\left(\mu_1, \mu_1\right) \ll(a, b)<(\alpha, \beta)$, then the equation
\begin{equation}\label{ee22}
\left\{\begin{array}{l}
L u=-\Delta u-\frac{1}{2}(x \cdot \nabla u)=a u^{+}-b u^{-}, \\
u\in X,
\end{array}\right.
\end{equation}
only has a trivial solution.
\end{lemma}
\begin{proof}
  Since $\mathcal{C}$ is symmetric, we only need to consider the case  $\alpha \geqslant \beta$. Without loss of generality, assume that $a<\alpha$. By contradiction, if we assume that $u_0$ is a nontrivial solution of (\ref{ee22}), then $u_0$ changes sign in $\mathbb{R}^N$ by Corollary \ref{le77}. Let $p=\alpha-\beta \geqslant 0$. Then, $\beta=c(p)$, where $c(p)$ is given by (\ref{ee16}). We show that there exists a path $\gamma \in \Gamma$ such that
$$
\max _{t \in[-1,1]} \tilde{I}_p(\gamma(t))<\beta,
$$
which contradicts the definition of $c(p)$.

In order to construct $\gamma$, we first prove that
\begin{equation}\label{ee23}
 \frac{\left\|u_0^{+}\right\|^2}{\left|u_0^{+}\right|_{K,2}^2}<\alpha,
\frac{\left\|u_0^{-}\right\|^2}{\left|u_0^{-}\right|_{K,2}^2} \leqslant \beta.
\end{equation}
In fact, by (\ref{ee22}), we know that
$$
\left\|u_0^{+}\right\|^2=a\left|u_0^{+}\right|_{K,2}^2<\alpha\left|u_0^{+}\right|_{K,2}^2,
$$
and
$$
\left\|u_0^{-}\right\|^2=b\left|u_0^{-}\right|_{K,2}^2 \leqslant \beta\left|u_0^{-}\right|_{K,2}^2.
$$
Hence, (\ref{ee23}) holds.
 Then, similar to the proof of Theorem \ref{t2}, but by using $u_0$ instead of $u$, there are two paths $\gamma_1, \gamma_3$ in $\mathcal{P}$,
as given by (\ref{ee211}), (\ref{ee212}),
which go from $u_0 /\left|u_0\right|_{K,2}$ to $u_0^{-} /\left|u_0^{-}\right|_{K,2}$,
and from $-u_0^{-} /\left|u_0^{-}\right|_{K,2}$ to $u_0 /\left|u_0\right|_{K,2}$, respectively.
The energy is lower than $\beta$ on the paths given above according to (\ref{ee23}).
Denote $O=\left\{u \in \mathcal{P}: I_p(u)<\beta-p\right\}$.
Then, we know that the only possible critical points of $I_p$ in $O$ are $-\phi_1$ when $\lambda_1<\beta-p$, and $\phi_1$.
Indeed, let $v \in O$ be a critical point of $I_p$ on $\mathcal{P}$ and write $c=I_p(v)$.
Then, from Theorem \ref{t2}, it follows that $c \leqslant \lambda_1$. If $c<\lambda_1$, then $c+p=\lambda_1$ by Corollary \ref{le77} and Corollary \ref{le78}.
 Hence, we obtain $v=\phi_1$ from Corollary \ref{le77}.
 In the following, we assume that $c=\lambda_1$.
 According to $\left(\lambda_1+p, \lambda_1\right) \in \mathcal{L}$ and Corollary \ref{le77},
 we find that $v= \pm \phi_1$ when $p=0$, and $v=-\phi_1$ when $p>0$.
 By Lemma \ref{le9} and Lemma \ref{le10}, there is a path $\gamma_2$ in $O$  that goes from $-u_0^{-} /\left|u_0^{-}\right|_{K,2}^2$ to $\phi_1$ or $-\phi_1$.
 We assume that this point is $\phi_1$.
 Thus, similar to Theorem \ref{t2}, we obtain the desired path $\gamma \in \Gamma$ and the energy is lower than $\beta$ along $\gamma$.

\end{proof}

\begin{proof} [Proof of Theorem \ref{t42}]
Let $p_1<p_2$. By contradiction, suppose that either $p_1+c\left(p_1\right) \geqslant p_2+c\left(p_2\right)$ or $c\left(p_1\right)=c\left(p_2\right)$.
In the first case, we have

$$
p_1+c\left(p_1\right) \geqslant p_2+c\left(p_2\right)>p_1+c\left(p_2\right),
$$
which implies $c\left(p_1\right)>c\left(p_2\right)$. Let $(\alpha, \beta)=\left(p_1+c\left(p_1\right), c\left(p_1\right)\right)$
in Lemma \ref{le41}. Then, we see that
\begin{equation*}
\left\{\begin{array}{l}
L u=-\Delta u-\frac{1}{2}(x \cdot \nabla u)=(p_2+c(p_2)) u^{+}-c(p_2) u^{-}, \\
u\in X,
\end{array}\right.
\end{equation*}
only has a trivial solution, which contradicts $\left(p_2+c\left(p_2\right), c\left(p_2\right)\right) \in \Sigma$. In the other case, we know that $p_1+c\left(p_1\right)<p_2+c\left(p_2\right)$. By taking $(\alpha, \beta)=\left(p_2+c\left(p_2\right), c\left(p_2\right)\right)$ in Lemma \ref{le41},
 we see that
\begin{equation*}
\left\{\begin{array}{l}
L u=-\Delta u-\frac{1}{2}(x \cdot \nabla u)=(p_1+c(p_1)) u^{+}-c(p_1) u^{-}, \\
u\in X,
\end{array}\right.
\end{equation*}
only has a trivial solution, which contradicts $\left(p_1+c\left(p_1\right), c\left(p_1\right)\right) \in \Sigma$. The proof is complete
\end{proof}

\begin{lemma}\label{l42}
   There exists $\varphi \in X$ such that there does not exist $r \in \mathbb{R}$
   verifying $\phi(x) \leqslant r \phi_1(x)$ a.e. in $\mathbb{R}^N$.
\end{lemma}

\begin{proof}
   When $ N\geq 2$, it suffices to take for $\phi$ a function in $X$ which is unbounded from above in the neighborhood of
    some $x_1 \in \mathbb{R}^N$.

     For $N=1$, using the Sobolev embedding theorem in \cite{bur6}  we have  $\exp (|x|^2 / 8) \phi \in C^{0,1 / 2}(\mathbb{R})$.
    Let $$\phi=\frac{1}{1+x^4}\exp (-|x|^2 / 8).$$ According to Lemma \eqref{le0}, $\phi_1=\exp \left(-|x|^2 / 4\right)$. Then $\phi$ and $\phi_1$ satisfy the requirement.
\end{proof}

\begin{proof} [Proof of Theorem \ref{t43}]
Assume by contradiction that there exists $\delta>0$ such that
$$
\max _{t \in[-1,1]} \tilde{I}_p(\gamma(t)) \geqslant \lambda_1+\delta
$$
for all $\gamma \in \Gamma$ and all $p \in \mathbb{R}_{+}$. Moreover, there exists $\phi \in X$
 that satisfies the result of Lemma \ref{l42}. Consider a path defined by
$$
\gamma(t)=\frac{t \phi_1+(1-|t|) \phi}{\left|t \phi_1+(1-|t|) \phi\right|_{K,2}}, \quad t \in[-1,1].
$$
For each $p \in \mathbb{R}_{+}$, suppose that $\tilde{I}_p\left(\gamma\left(t_p\right)\right)=\max _{t \in[-1,1]} \tilde{I}_p(\gamma(t))$ for some $t_p \in[-1,1]$,
and let $u_p= t_p \phi_1+\left(1-\left|t_p\right|\right) \phi$. Then, we obtain
\begin{equation}\label{ee24}
 \left\|u_p\right\|^2-p\left|u_p^{+}\right|_{K,2}^2 \geqslant\left(\lambda_1+\delta\right)\left|u_p\right|_{K,2}^2
\end{equation}
for all $p \in \mathbb{R}_{+}$. If we let $p \rightarrow \infty$ in (\ref{ee24}), then we can assume that there exist a sequence $\left\{p_n\right\}$
with $p_n \rightarrow \infty$ and $t_{p_n} \rightarrow t_{\infty} \in[-1,1]$. Since $\left\{u_{p_n}\right\}$ is bounded in $X$, from (\ref{ee24}),
it follows that $\left|u_{p_n}^{+}\right|_{K,2}^2 \rightarrow 0$. Hence,
$$
\int_{\mathbb{R}_{+}}\left(\left[t_{\infty} \phi_1+\left(1-\left|t_{\infty}\right|\right) \phi\right]^{+}\right)^2K(x)dx=0
$$
which means that $t_{\infty} \phi_1(x)+\left(1-\left|t_{\infty}\right|\right) \phi(x) \leqslant 0$, a.e. $x \in \mathbb{R}_{+}$.
Based on the choice of $\phi$, we find that $\left|t_{\infty}\right|=1$, and thus $t_{\infty}=-1$ by the fact that
$\phi_1>0$, i.e., $t_{p_n} \rightarrow-1$ and $u_{p_n} \rightarrow-\phi_1$.
Thus, from (\ref{ee24}) again, it follows that
$$
\lambda_1\left|\phi_1\right|_{K,2}^2=\left\|\phi_1\right\|^2 \geqslant\left(\lambda_1+\delta\right)\left|\phi_1\right|_{K,2}^2
$$
which is a contradiction.

\end{proof}

\section{Application}
In this section we study the solvability of   problem \eqref{e51}.
Let $F(x, t)=\int_0^t f(x, s) \mathrm{d} s$. We impose the following conditions on the nonlinearity $f(x, u)$ :
\begin{description}
  \item[($f_1$)] $f \in C\left(\mathbb{R}^N \times \mathbb{R}, \mathbb{R}\right)$ and  there exists $p\in (2,2^{\ast})$ such that
  $$|f(x, t)| \leq C\left(1+|t|^{p-1}\right)$$ for $(x, t) \in \mathbb{R}^N \times \mathbb{R}$,
 where  $2^{\ast}= 2N/(N -2)$ if $N > 2$, and $2^{\ast}=\infty $ if $N \leq 2$.
  \item[($f_2$)] $f(x, t) \geqslant 0$ for all
$x \in \mathbb{R}^N, t>0$, and $f(x, t)=0$ for all $x \in \mathbb{R}^N, t\leq 0$.
  \item[($f_3$)] There exist $f_0, f_{\infty} \in (\lambda_1, \infty)$ such that
$$\lim _{t \rightarrow 0^{+}} \frac{f(x, t)}{t}=f_0$$ and $$\lim _{t \rightarrow \infty} \frac{f(x, t)}{t} =f_{\infty}$$ uniformly for $x \in \mathbb{R}^N$.
\end{description}

Now, we have the following theorem.

\begin{theorem}\label{t51}
 Assume that $\left(f_{1}\right)$,$\left(f_2\right)$ and $\left(f_3\right)$ are satisfied, then the problem (\ref{e51}) has at least two positive solutions in $X$.
\end{theorem}
By using the above result, we can prove that the functional $I(u):  X \rightarrow \mathbb{R}$
given by
$$
I(u)=\frac{1}{2} \int_{\mathbb{R}^N} K(x)|\nabla u|^2 d x-\int_{\mathbb{R}^N} K(x) F(x, u) d x
$$
is well defined. Standard calculations  imply that $I \in C^1(X, \mathbb{R})$ and the derivative of $I$ at the point $u$ is given by
$$
\left\langle I^{\prime}(u), v\right\rangle=\int_{\mathbb{R}^N} K(x) \nabla u \cdot \nabla v d x-\int_{\mathbb{R}^N} K(x) f(x, u) v d x
$$
for any $v \in X$.  Hence, the critical points of $I$ are precisely the weak solutions of problem \eqref{e51}.

The following preliminary results are used in the proof of Theorem \ref{t51}. Firstly, we prove that the functional
$I$ satisfies the mountain pass geometry.

\begin{lemma}\label{ll51}
Suppose that ($f_1$), ($f_2$) and ($f_3$) hold. Then  there exist $\rho>0$ and $\alpha>0$ such that

(i) $I(u) \geqslant \alpha$ for all $\|u\|=\rho$;

(ii) there exists $e \in X$ such that $\|e\|>\rho$ and $I(e)<0$.
\end{lemma}
\begin{proof}
(i) For all $u\in X$, according to ($f_1$) and ($f_3$),  there exist $\varepsilon$ and $C_1>0$ such that
$$
|F(t)| \leqslant \frac{(f_0+\varepsilon)}{2} t^2+C_1|t|^p, \quad \text{\ for\ all\ } t \in \mathbb{R}.
$$
Hence, using Lemma \ref{le01}, we have
$$
\begin{aligned}
I(u) & \geqslant \frac{1}{2}\|u\|^2-\frac{1}{2}(f_0+\varepsilon) \int_{\mathbb{R}^N}  u^2 K(x) dx-C_1\int_{\mathbb{R}^N}  u^p K(x)dx \\
& \geqslant \frac{1}{2}\|u\|^2-\frac{1}{2}(f_0+\varepsilon) S^2_2\|u\|^2-C_1 S_{K,p}^p\|u\|^p \\
& =\frac{1}{2}\|u\|\left(\|u\|-a-b\|u\|^{p-1}\right),
\end{aligned}
$$
where $a= (f_0+\varepsilon) S^2_2$ and $b=2 C_1 S_{K,p}^p$.

Set $h: \mathbb{R}_{+} \rightarrow \mathbb{R}$ by $h(t)=t-a-b t^{p-1}$ for all $t \in \mathbb{R}_{+}$. Then, we get
$$
\max _{t \in \mathbb{R}_{+}} h(t)=h\left(t_0\right)=\left[\frac{1}{b(p-1)}\right]^{1 /(p-2)} \frac{p-2}{p-1}-a,
$$
where $t_0=[b(p-1)]^{-1 /(p-2)}>0$, which implies that $h\left(t_0\right)>0$ if and only if $a^{p-2} b<(p-2)^{p-2} /(p-1)^{p-1}$.
By taking $\rho=t_0$, we have that
$$I(u) \geqslant 2^{-1} t_0 h\left(t_0\right):=\alpha>0$$ for all $u \in X$ and $\|u\|=\rho$.

(ii) Using ($f_3$), there exist $\delta_1, M_1>0$ such that
\begin{equation}\label{e52}
2 F(x, t) \geq\left(\lambda_1+\delta_1\right) t^2, \quad \forall x \in \mathbb{R}^N, \forall t \geq M_1.
\end{equation}
Take $R_1>0$ large enough such that
\begin{equation}\label{e53}
\int_{ B_{R_1}}|\phi_1|^2 K(x)dx\geq \frac{\lambda_1+\frac{1}{2} \delta_1}{\lambda_1+\delta_1}\left|\phi_1\right|_{K,2}^2.
\end{equation}
Since $\phi_1(x)>0$ in $\mathbb{R}^N$, there exists $t_1>0$ such that
$$
t_1 \phi_1(x)>M_1, \quad \text{\ for\ all\ }|x| \leq R_1 .
$$
Then, by (\ref{e52}), (\ref{e53}) and Lemma \ref{le0}, we can see that, for $t \geq t_1$,
$$
\begin{aligned}
I\left(t \phi_1\right) & =\frac{t^2}{2} \int_{\mathbb{R}^N}  |\nabla \phi_1|^2 K(x) d x-\int_{\mathbb{R}^N}  F\left(x, t \phi_1\right) K(x) d x \\
& =\frac{t^2}{2} \lambda_1\left|\phi_1\right|_{K,2}^2-\int_{B_{R_1}}  F\left(x, t \phi_1\right) K(x) d x-\int_{B_{R_1}^c}  F\left(x, t \phi_1\right) K(x) d x \\
& \leq \frac{t^2}{2} \lambda_1\left|\phi_1\right|_{K,2}^2-\frac{t^2}{2}\left(\lambda_1+\delta_1\right) \int_{B_{R_1}} \phi_1^2 K(x)d x \\
& \leq-\frac{t^2}{4} \delta_1\left|\phi_1\right|_{K,2}^2.
\end{aligned}
$$
Take $e=\tilde{t} \phi_1$ with $\tilde{t}>t_1$. It is easily seen that $I(e)<0$.
\end{proof}

We recall that a sequence $\left\{u_n\right\} \subset X$ is said to be a $(C)_c$ sequence if $I\left(u_n\right) \rightarrow c$ and $\left(1+\left\|u_n\right\|\right) I^{\prime}\left(u_n\right) \rightarrow 0$.
The functional $I$ is said to satisfy the $(C)_{c}$ condition if any $(C)_c$ sequence of $I$ has a convergent subsequence.
If we replace the PS condition in the mountain pass theorem with the $(C)_c$ condition, the mountain pass theorem remains valid.
Hence, to obtain critical points of the functional $I$, we need to show that $I$ satisfies the $(C)_c$ condition.
\begin{lemma}\label{le53}
Suppose that  ($f_2$) and ($f_3$) hold. Then the functional $I$ satisfies the $(C)_c$ condition for any $c \in \mathbb{R}$.
\end{lemma}
\begin{proof}
 Let $\left\{u_n\right\} \subset X$ such that
\begin{equation}\label{e54}
I\left(u_n\right) \rightarrow c, \quad\left(1+\left\|u_n\right\|\right)\left\|I^{\prime}\left(u_n\right)\right\| \rightarrow 0 .
\end{equation}
We suppose by contradiction that $\left\|u_n\right\| \rightarrow \infty$. Set $w_n=\frac{u_n}{\| u_n\|} $. Then $\left\|w_n\right\|=1$, and there exists $w_0 \in X$ such that, passing to a subsequence if necessary,
$$
\begin{gathered}
w_n \rightharpoonup w_0 \text { in } X, \\
w_n \rightarrow w_0 \text { in } L_K^2\left(\mathbb{R}^N\right), \\
w_n(x) \rightarrow w_0(x) \text { for a.e. } x \in \mathbb{R}^N .
\end{gathered}
$$
We prove that $w_0(x) \not \equiv 0$. In fact, if not, we assume that $w_0(x) \equiv 0$,
that is, $w_n \rightarrow 0$ in $L_K^2\left(\mathbb{R}^N\right)$. From (\ref{e54}) and ($f_2$) we deduce
$$
\begin{aligned}
o(1) & =\frac{\left\langle I^{\prime}\left(u_n\right), u_n\right\rangle}{\left\|u_n\right\|} \\
& =\int_{\mathbb{R}^N}  \nabla w_n^2 K(x) d x-\int_{\mathbb{R}^N}  \frac{f\left(x, u_n\right)}{\left\|u_n\right\|} w_n K(x) d x \\
& =1-\int_{\mathbb{R}^N}  \frac{f\left(x, u_n\right)}{\left\|u_n\right\|} w_n^2 K(x) d x \\
& \rightarrow 1.
\end{aligned}
$$
which is a contradiction. Choosing $w_n^{-}(x)=\max \left\{-w_n(x), 0\right\}$ as test function, using (\ref{e54}) again we have
$$
\begin{aligned}
o(1) & =\frac{\left\langle I^{\prime}\left(u_n\right), w_n^{-}\right\rangle}{\left\|u_n\right\|} \\
& =\int_{\mathbb{R}^N} K(x) \nabla\left(w_n^{-}\right)^2 d x+\int_{\mathbb{R}^N} K(x) \frac{f\left(x, u_n\right)}{\left\|u_n\right\|} w_n^{-} d x \\
& =\left\|w_n^{-}\right\|^2+\int_{\mathbb{R}^N} K(x) \frac{f\left(x, u_n\right)}{\left\|u_n\right\|}\left(w_n^{-}\right)^2 d x.
\end{aligned}
$$
From ($f_2$) it follows that $w_n^{-}(x)=o(1)$. Then we have $w_0^{-}(x)=0$ for a.e. $x \in \mathbb{R}^N$ and thus $w_0(x) \geq 0$ and $w_0(x) \not \equiv 0$.

Let
$$
A_1=\left\{x \in \mathbb{R}^N: w_0(x)=0\right\}, \quad A_2=\left\{x \in \mathbb{R}^N: w_0(x)>0\right\} .
$$
If $x \in A_1$, using ($f_2$), ($f_3$), there exists $C_1>0$ such that

\begin{equation}\label{e55}
0 \leq \frac{f(x, t)}{t} \leq C_1, \quad \forall x \in \mathbb{R}^N, \forall t \in \mathbb{R},
\end{equation}
which implies that
$$
\frac{\left|f\left(x, u_n(x)\right)\right|}{\left\|u_n\right\|}=\left|\frac{f\left(x, u_n(x)\right)}{u_n(x)} w_n(x)\right| \leq C_1\left|w_n(x)\right| \rightarrow 0 \text { for a.e. } x \in A_1 .
$$
Then we have

\begin{equation}\label{e56}
\frac{f\left(x, u_n(x)\right)}{\left\|u_n\right\|} \rightarrow 0=\left(\lambda_1+\delta_1\right) w_0(x) \text { for a.e. } x \in A_1,
\end{equation}
where $\delta_1$ is taken as in (\ref{e52}).

If $x \in A_2$, then $u_n(x)=w_n(x)\left\|u_n\right\| \rightarrow+\infty$. Using ($f_3$) we obtain
\begin{equation}\label{e57}
\liminf _{n \rightarrow+\infty} \frac{f\left(x, u_n(x)\right)}{\left\|u_n\right\|}=\liminf _{n \rightarrow \infty} \frac{f\left(x, u_n(x)\right)}{u_n(x)} w_n(x) \geq\left(\lambda_1+\delta_1\right) w_0(x) \text { for a.e. } x \in A_2.
\end{equation}
According to $w_n \rightharpoonup w_0$ in $X$, it follows that $\int_{\mathbb{R}^N} w_n \phi_1 K(x)d x \rightarrow \int_{\mathbb{R}^N} w_0 \phi_1 K(x)d x$. By (\ref{e54}) and Lemma \ref{le0}, we have
$$
\begin{aligned}
o(1) & =\frac{\left\langle I^{\prime}\left(u_n\right), \phi_1\right\rangle}{\left\|u_n\right\|} \\
& =\int_{\mathbb{R}^N} \nabla w_n \cdot \nabla \phi_1  K(x) d x-\int_{\mathbb{R}^N}  \frac{f\left(x, u_n\right)}{\left\|u_n\right\|} \phi_1  K(x) d x \\
& =\lambda_1 \int_{\mathbb{R}^N} w_n \phi_1 K(x)d x-\int_{\mathbb{R}^N} \frac{f\left(x, u_n\right)}{u_n} w_n \phi_1  K(x) d x.
\end{aligned}
$$
Then, from (\ref{e56}), (\ref{e57}) and Fatou's lemma, it follows that

$$
\lambda_1 \int_{\mathbb{R}^N} w_0 \phi_1 K(x)d x \geq\left(\lambda_1+\delta_1\right) \int_{\mathbb{R}^N} w_0 \phi_1 K(x)d x.
$$
Since, $\phi_1(x)>0$ and $w_0 \geq 0, w_0 \not \equiv 0$, we have  $\int_{\mathbb{R}^N} w_0 \phi_1 K(x) d x>0$. Hence $\lambda_1 \geq \lambda_1+\delta_1$, a contradiction.
This imlies that $\left\{u_n\right\}$ is bounded in $X$.  Thus there exists $u_0 \in X$ such that, passing to a subsequence if necessary,
$$
\begin{gathered}
u_n \rightharpoonup u_0 \text { in } X, \\
u_n \rightarrow u_0 \text { in } L_K^2\left(\mathbb{R}^N\right), \\
u_n(x) \rightarrow u_0(x) \text { for a.e. } x \in \mathbb{R}^N .
\end{gathered}
$$
Using  ($f_2$) and ($f_3$), it is easy to obtain that
$$
\int_{\mathbb{R}^N} f\left(x, u_n\right)\left(u_n-u_0\right) K(x) d x=o(1).
$$
Noting that $\left\langle I^{\prime}\left(u_n\right), u_n-u_0\right\rangle \rightarrow 0$, it follows that
$$
\left\|u_n-u_0\right\|^2=\int_{\mathbb{R}^N}  f\left(x, u_n\right)\left(u_n-u_0\right) K(x) d x+o(1)=o(1),
$$
which implies that $u_n \rightarrow u_0$ in $X$. This completes the proof.
\end{proof}

\begin{proof} [Proof of Theorem \ref{t51}]
According to ($f_2$) and ($f_3$), for any given $\varepsilon \in\left(0, f_0-\lambda_1\right)$,  such that
  there exist $\delta_1>0$ such that

 $$
f(x, t) \geqslant\left(f_0-\varepsilon\right) t, \quad \forall x \in \mathbb{R}^N, \forall 0<t \leq \delta_1.
$$
Thus, we have
$$
F(x, t) \geqslant \frac{1}{2}\left(f_0-\varepsilon\right) t^2, \forall x \in \mathbb{R}^N, \forall 0 <t \leq \delta_1.
$$
Since $\phi_1(x)>0$ in $\mathbb{R}^N$, there exists $t_1>0$ such that
$$
t_1 \phi_1(x)<\delta_1, \quad \text{\ for\ all\ } x\in \mathbb{R}^N .
$$
Hence, we infer
$$
\begin{aligned}
I\left(t \phi_1\right) & =\frac{t^2}{2} \int_{\mathbb{R}^N}  |\nabla \phi_1|^2 K(x) d x-\int_{\mathbb{R}^N}  F\left(x, t \phi_1\right) K(x) d x \\
& \leq \frac{t^2}{2} \lambda_1\left|\phi_1\right|_{K,2}^2-\frac{t^2}{2}\left(f_0-\varepsilon\right)\int_{\mathbb{R}^N}  \phi^2 K(x) d x<0
\end{aligned}
$$
for all $t\in (0,t_1)$.
 By Lemma \ref{ll51}, there exist $\rho, \alpha>0$ such that $I(u) \geqslant \alpha$ for all $u \in X$ with $\|u\|=\rho$.
 Hence, we have $d=\inf \left\{I(u): u \in \bar{B}_\rho\right\}<0$. Then, there is a minimizing sequence $\left\{u_n\right\} \subset \bar{B}_\rho$  such that
$$
I\left(u_n\right) \rightarrow d, \quad I^{\prime}\left(u_n\right) \rightarrow 0
$$
by Ekeland's variational principle. From Lemma \ref{le53}, there is  a $u_1 \in \bar{B}_\rho$  such that $I\left(u_1\right)=d<0$
and $I^{\prime}\left(u_1\right)=0$, which shows that $u_1$ is a nontrivial critical point of $I$.

Moreover, by the mountain pass theorem \cite{bur7} for
$$
c=\inf _{\gamma \in \Gamma} \max _{t \in[0,1]} I(\gamma(t))
$$
and
$$
\Gamma=\{\gamma \in C([0,1], X): \gamma(0)=0, \gamma(1)=e\},
$$
there exists $u_2 \in X$ such that $I\left(u_2\right)=c>0$ and $I^{\prime}\left(u_2\right)=0$, which implies that $u_2$ is another critical point of $I$.

Finally, we show that $u_1, u_2$ are positive solutions of the problem (\ref{e51}). In fact, if $u$ is a critical point of $I$,
then after multiplying  (\ref{e51})  by $u^{-}$,
we see that $\left\|u^{-}\right\|^2=0$, which means that  $u^{-}=0$, and thus $u \geqslant 0$,
as a result, $u_1>0$ and $u_2>0$ by a standard process. The proof is complete.

\end{proof}

\end{document}